\documentclass[]{article}
\usepackage{latexsym,amssymb,amsmath,amsthm}
\usepackage[dvips]{graphics}

\newtheorem{thm}{Theorem}

\newtheorem{lemma}[thm]{Lemma}
\newtheorem{propo}[thm]{Proposition}
\theoremstyle{definition}

\newtheorem{remark}{Remark}

\newcommand{\C}{\mathbb{C}}
\newcommand{\R}{\mathbb{R}}
\newcommand{\Z}{\mathbb{Z}}
\newcommand{\T}{\mathbb{T}}
\newcommand{\floor}[1]{\left\lfloor #1 \right\rfloor}
\newcommand{\comment}[1]{}
\newcommand{\abs}[1]{\left\lvert #1 \right\rvert}
\DeclareMathOperator{\Var}{Var}

\title{Analytic continuation of Dirichlet series with almost periodic coefficients}
\author{Oliver Knill and John Lesieutre 
}
\date{November 9, 2008}

\begin{document}
\maketitle
\begin{abstract}
  We consider Dirichlet series
  $\zeta_{g,\alpha}(s)=\sum_{n=1}^\infty g(n\alpha) e^{-\lambda_n s}$ 
  for fixed irrational $\alpha$ and periodic functions $g$. We
  demonstrate that for Diophantine $\alpha$ and smooth $g$, the
  line ${\rm Re}(s)=0$ is a natural boundary in the Taylor series
  case $\lambda_n=n$, so that the unit circle is the maximal domain of
  holomorphy for the almost periodic Taylor series
  $\sum_{n=1}^{\infty} g(n\alpha) z^n$. We prove that a
  Dirichlet series $\zeta(s) = \sum_{n=1}^{\infty} g(n \alpha)/n^s$ has
  an abscissa of convergence $\sigma_0 = 0$ if $g$ is odd and real analytic
  and $\alpha$ is Diophantine. We show that if $g$ is odd and has bounded 
  variation and
  $\alpha$ is of bounded Diophantine type $r$, the abscissa of convergence
  is smaller or equal than $1-1/r$. Using a polylogarithm
  expansion, we prove that if $g$ is odd and real analytic and $\alpha$ is
  Diophantine, then the Dirichlet series $\zeta(s)$ 
  has an analytic continuation to the entire complex plane. 
\end{abstract}

\vspace{8mm}

\section{Introduction} 
Let $g : \R \to \C$ be a piecewise continuous $1$-periodic $L^2$ function with
Fourier expansion $g(x) = \sum_{k=-\infty}^\infty c_k e^{2 \pi i kx}$. 
Define the $\zeta$ function 
$$ \zeta_{g,\alpha}(s) = \sum_{n=1}^{\infty} \frac{g(n \alpha)}{n^s} \; . $$
For irrational $\alpha$, we call this a
{\bf Dirichlet series with almost periodic coefficients}. 
An example is the Clausen function, where 
$g(x) = \sin(2 \pi x)$ or the poly-logarithm, where $g(x) = \exp(2 \pi i
x)$. Another example arises with $g(x)=x-\floor{x+1/2}$, the signed
distance from $x$ to the nearest integer. Obviously for ${\rm Re}(s)>1$, 
such a Dirichlet series converges uniformly to an analytic limit.

The case that \(\alpha\) is rational is less interesting, as the
following computation illustrates.

{\it 
For periodic $\alpha=p/q$ and any odd function $g$, the zeta function
has an abscissa of convergence $0$ and allows an analytic continuation 
to the entire plane.
}
\begin{proof}
Write
\begin{eqnarray*}
\sum_{n=1}^\infty \frac{g(np/q)}{n^s} 
 &=& \sum_{\ell=1}^q \sum_{n=0}^\infty \frac{g((nq+\ell)(p/q))}{(nq+\ell)^s}  \\
 &=& \frac{1}{q^s} \sum_{\ell=1}^q \sum_{n=0}^\infty 
      \frac{g(n+\ell p/q)}{(n+\ell/q)^s} 
  = \frac{1}{q^s} \sum_{\ell=1}^q g(\ell/q) \zeta(s,\ell/q) \; , 
\end{eqnarray*}
where $\zeta(s,u)=\sum_{n=1}^{\infty} 1/(n+u)^s$ is the Hurwitz zeta function. So,
the periodic zeta function is a just finite sum of Hurwitz zeta
functions which individually allow a meromorphic continuation. Each
Hurwitz zeta function is analytic everywhere except at $1$, where it has
a pole of residue $1$: the series $\zeta(s,u) - 1/(s-1)$ has the 
abscissa of convergence $0$ and allows an analytic extension to the plane.
So, if $\sum_{n=1}^q g(n \alpha) = 0$, which 
is the case for example if $g$ is odd, then the periodic 
Dirichlet series has abscissa of convergence $0$ and 
admits an analytic continuation to the plane.
\end{proof}
When \(\alpha = 0\), the function \(\zeta_{g,\alpha}\) is merely a
multiple of the Riemann zeta function. \\

An other special case  $\lambda_n=n$ leads to Taylor series 
$f(z) = \sum_{n=1}^{\infty} a_n z^n$ 
with $z=e^{-s}$. Also here, the rational case is well understood: \\

{\it If $\alpha=p/q$ is rational, the function 
$f(z) = \sum_{n=1}^{\infty} a_n z^n$ has a meromorphic
extension to the entire plane. }

\begin{proof}
  If $\alpha=p/q$ define $h(z) = \sum_{n=1}^q g(n \alpha) z^n$.  Then
  $$  f(z) = h(z) (1+z^q+z^{2q} + \dots) = \frac{h(z)}{1-z^q} \; . $$
  The right hand side provides the meromorphic continuation of $f$. 
\end{proof}

We usually assume $\int g \; dm =0$ because we are interested in the 
growth of the random walk in the case $s=0$ and because if 
$\int g \; dm \neq 0$, the abscissa of convergence is in general $1$: 
for $f(x) = a_0+\sin(x)$ for example, where the Dirichlet series with 
$\lambda_n=\log(n)$ is a sum of the standard zeta function and the 
Clausen function. The abscissa of convergence is $\sigma_0=1$
except for $a_0=0$, where the abscissa drops to $0$.   \\

Zeta functions $\sum_{n=1}^{\infty} a_n e^{-\lambda_n s}$ can more
generally be considered for any dynamically generated sequence $a_n =
g(T^nx)$, where $T$ is a homeomorphism of a
compact topological space $X$ and $g$ is a continuous function and
$\lambda_n$ grows monotonically to $\infty$.  \\

Random Taylor series associated
with an ergodic transformation were considered in \cite{Halchin, Ionescu}. 
The topic has also been explored in a probabilistic setup,
where $a_n$ are independent random symmetric variables, in which case
the line ${\rm Re}(s) = \sigma_0$ is a natural boundary \cite{Kahane}.
Analytic continuation questions have also been studied for other
functions: if the coefficients are generated by finite automata, a
meromorphic continuation is possible \cite{AFP}. \\

In this paper we focus on Taylor series and ordinary Dirichlet series.
We restrict ourselves to the case, where the dynamical system is an
irrational rotation $x \mapsto x+\alpha$ on the circle. The minimality
and strict ergodicity of the system will often make the question
independent of the starting point $x \in X$ and allow to use
techniques of Fourier analysis and the Denjoy-Koksma inequality. We
are able to make statements
if $\alpha$ is Diophantine. \\

Dirichlet series can allow to get information on the growth of
the random walk $S_k = \sum_{n=1}^k g(T^n(x))$ for a $m$-measure
preserving dynamical system $T:X \to X$ if $\int_X g(x) \; dm(x) =0$.
Birkhoffs ergodic theorem assures $S_k = o(k)$. Similarly as the law of iterated
logarithm refines the law of large numbers in probability theory, 
and Denjoy-Koksma type results provide further estimates on the growth rate
in the case of irrational rotations, one can study the growth rate for more
general dynamical systems. 
The relation with algebra is as follows: if $S_k$ grows like $k^{\beta}$
then the abscissa of convergence of the ordinary Dirichlet series is
smaller or equal to $\beta$. In other words, establishing
bounds for the analyticity domain allows via Bohr's theorem to get
results on the abscissa of convergence which give upper bounds
of the growth rate. Adapting the $\lambda_n$ to the situation allows to 
explore different growth behavior. The algebraic concept of Dirichlet 
series helps so to understand a dynamical concept.  

\section{Almost periodic Taylor series}
A general Dirichlet series is of the form
$$ \sum_{n=1}^{\infty} a_n e^{-\lambda_n s}   \; . $$
For $\lambda_n=\log(n)$ this is an ordinary Dirichlet series, while
for $\lambda_n=n$, it is a Taylor series $\sum_n a_n z^n$ with
$z=e^{-s}$.  We primarily restrict our attention to these two cases.  \\

We begin by considering the easier problem of Taylor series with
almost periodic coefficients and examine the analytic continuation of
such functions beyond the unit circle. Given a non-constant periodic
function $g$, we can look at the
problem of whether the Taylor series
$$  f(z) = \sum_{n=1}^{\infty} g(n \alpha) z^n  $$
can be analytically continued beyond the unit circle.
Note that all these functions have radius of convergence $1$
because $\limsup_n |g(n \alpha)|^{1/n}=1$. 

We have already seen in the introduction that if 
$\alpha=p/q$ is rational, the function $f$ has a meromorphic
extension to the entire plane. There is an other case where
analytic continuation can be established immediately:

\begin{lemma}
\label{taylortrig}
If $g$ is a trigonometric polynomial and $\alpha$ is arbitrary, 
then $f$ has a meromorphic extension to the entire plane. 
\end{lemma}
\begin{proof}
  Since $g(x) = \sum_{n=1}^k c_n e^{2 \pi i n x}$, it is enough to
  verify this for $g(x) = e^{2 \pi i n x}$, in which case the series
  sums to $f(z) = 1/(1-e^{2 \pi i n \alpha} z)$.
\end{proof}

On the other hand, if infinitely many of the Fourier coefficients for
\(g\) are nonzero, analytic continuation may not be possible.
\begin{propo}
  \label{propo1}
  Fix $r>1$. Assume $g$ is in $C^{t}$ for $t>2r+1$, 
  and that all Fourier coefficients $c_k$ of $g$
  are nonzero and that $\alpha$ is of Diophantine type $r$. Then the
  almost periodic Taylor series $f_{g,\alpha}(z) = \sum_{n=0}^\infty g(n\alpha) z^n$
  can not be continued beyond the unit circle.
\end{propo}
\begin{proof}
Write
\begin{eqnarray*}
f_{g,\alpha}(z) &=& \sum_{n=0}^\infty g(n\alpha) z^n  
      = \sum_{n=0}^\infty \sum_{k=-\infty}^\infty c_ke^{2 \pi i k n\alpha} z^n  \\
      &=& \sum_{k=-\infty}^\infty c_k \sum_{n=0}^\infty e^{2\pi i k n\alpha} z^n  
       = \sum_{k=-\infty}^\infty \frac{c_k}{1-e^{2 \pi i k\alpha} z} \; . 
\end{eqnarray*}
Fix some $j$ and consider the radial limit
\begin{eqnarray*}
  \lim_{t \to 1^-} f_{g,\alpha}( t e^{2 \pi i j \alpha} ) 
  &=& \lim_{t \to 1^-} \sum_{k=-\infty}^\infty 
  \frac{c_k}{1-t e^{2 \pi i k \alpha} e^{2 \pi i j \alpha}} \\
  &=& \lim_{t \to 1^-} \left( \frac{c_j}{1-t} 
    + \sum_{\substack{k=-\infty \\ k \neq -j}}^\infty
    \frac{c_k}{1-t e^{2 \pi i(k+j) t}} \right) \; . 
\end{eqnarray*}
The latter sum converges at $t=1$, as by the Diophantine
condition, the denominator is bounded below by $(n-k)^{2r}$, while the
numerator $c_k$ is bounded above by $1/k^t$ for 
$t>2r+1$ by the differentiability
assumption on $g$. Because the term $c_j/(1-t)$ diverges for $t \to 1^-$, 
it follows that this radial limit is infinite for all $j$. 
Consequently, $f$ does not admit an analytic continuation to any larger set.
\end{proof}

\begin{remark}
This result is related to a construction of Goursat, which shows
that for any domain $D$ in $\C$, there exists a function which
has $D$ as a maximal domain of analyticity \cite{remmert}. 
In contrary to lacunary Taylor series like $\sum_{j=1}^{\infty} z^{2^j}$ 
which have the unit circle as a natural boundary too, all Taylor 
coefficients are in general nonzero in the Taylor series 
of Proposition~\ref{propo1}.
\end{remark}

\begin{remark}
  The function 
  $$ f(z) = \sum_{k=-\infty}^\infty 
  \frac{c_k}{1-e^{2 \pi i k\alpha} z} 
  = \sum_{k=-\infty}^{\infty} \frac{a_k}{z-z_k} \;  $$
  is defined also outside the unit circle.  The subharmonic function 
  $\log(f)(z)=\int \log|z-w| \;
  dk(w)$ has a Riesz measure $dk$ supported on the unit circle which
  is dense pure point.
\end{remark}

\begin{remark}
  The requirement $g \in C^t$ is by no means
  necessary. For example, any nonconstant step function $g$ is not
  even continuous, but by Szeg\H{o}'s theorem (see ~\cite{remmert}) 
  for power series with finitely many distinct coefficients which do not eventually repeat
  periodically, $f(z) = \sum_{n=1}^{\infty} g(n \alpha) z^n$ 
  can not be analytically extended beyond the unit disk. 
\end{remark}

\begin{remark}
The condition that all Fourier coefficients are nonzero may be relaxed
to the assumption that the set of $e^{2 \pi i k \alpha}$, where $k$
ranges over the indices of nonzero Fourier coefficients, is dense in
$S^1$.  
\end{remark}

As the last remark may suggest, trigonometric polynomials are not
the only functions whose associated series allow an analytic
continuation beyond the unit circle.
\begin{propo}
  \label{propo2}
  Let $K \subset \{|z|=1 \}$ be an arbitrary closed set on the unit
  circle. There exists an almost periodic Taylor series which has an
  analytic continuation to $\C \setminus K$ but not to any
  point of $K$.
\end{propo}

\begin{proof}
Set 
$$ c_k = \begin{cases} 0 & \text{if $e^{2 \pi i\alpha k} \in K$,} \\
  1/k! & \text{otherwise} \; .
\end{cases} $$ and let $g(x) = \sum_{k=-\infty}^\infty c_k e^{2 \pi i
  k x}$, and let $\alpha$ be Diophantine of any type $r>1$.  Inside
the unit circle we have
$$ f_{g,\alpha}(z) = \sum_{k=-\infty}^\infty \frac{c_k}{1-e^{2 \pi i k \alpha} z}.  $$
For any $j$ for which $e^{2 \pi i j \alpha} \in K$, this sum converges
uniformly on a closed ball around $z$ which does not intersect $K$
(since the denominators of the non-vanishing terms are uniformly
bounded), and thus has an analytic neighborhood around such $e^{2 \pi
  i j \alpha}$. Any point in $K$ lies in a compact neighborhood of
such a point.  On the other hand, by the arguments of the preceding
lemma, analytic continuation is not possible in $K$ itself.
\end{proof}

\section{Ordinary Dirichlet series}

Cahen's formula for the abscissa of convergence of an ordinary
Dirichlet series $\zeta_{g,\alpha}(s) = \sum_{n=1}^{\infty} a_n/n^s$
is
$$ \sigma_0 = \limsup_k \frac{\log S_k}{\log k}, $$
if $S_k = \sum_{n=1}^k g(n\alpha)$ does not converge ~\cite{HardyRiesz}.
We will compute the abscissa of convergence
for two classes of functions and so derive bounds on the
random walks $S_k$ which are stronger than those implied by the
Denjoy-Koksma inequality. \\

The first situation applies to real analytic $g$, where we can invoke the 
cohomology theory of cocycles over irrational rotations: 

\begin{propo}
  \label{propo3}
  If $g$ is real analytic, $\int g(x) \; dx=0$ and $\alpha$ is
  Diophantine, then the series for $\zeta_{g,\alpha}(s)$ converges and
  is analytic for ${\rm Re}(s) >0$. In other words, the 
  abscissa of convergence is $0$.
\end{propo}
\begin{proof}
Since $g_0=0$, $g=\sum_{n=1}^{\infty} g_n e^{2 \pi i n x}$ is 
an additive coboundary: the Diophantine
property of $\alpha$ implies that the real analytic function 
$$ h(x) = \sum_{n=1}^{\infty} g_n \frac{e^{2 \pi i n x}}{(e^{2 \pi i n \alpha}-1)} $$ 
solves
$$  g(x) = h(x+\alpha)-h(x)  \; \; {\rm mod} \; 1\; . $$ 
Because $S_k = \sum_{n=1}^k g(n \alpha)$ does not converge, 
but stays bounded in absolute value by $2 ||h||_{\infty}$,
Cahen's formula immediately implies that $\sigma_0=0$.  \\
Lets give a second proof without Cahen's formula.
For $0 < {\rm Re}(s) < 1$, we have 
\begin{eqnarray*}
 \sum_{n=1}^{\infty} \frac{g(n\alpha)}{n^s} 
            &=& \sum_{n=1}^{\infty} \frac{h((n+1)\alpha)}{n^s} - \frac{h(n\alpha)}{n^s} \\
            &=& -h(\alpha) + \sum_{n=2}^{\infty} h(n\alpha) [1/(n-1)^s - 1/n^s]  \\
            &=& -h(\alpha) + \sum_{k=2}^{\infty} h(n\alpha) \frac{n^s-(n-1)^s}{n^s(n-1)^s}  \\
            &\leq& ||h|| (1 + \sum_k \frac{n^s-(n-1)^s}{n^s(n-1)^s})  \; . 
\end{eqnarray*}
The sum is bounded for $1>{\rm Re}(s)>0$ because $|(n+1)^s - n^s|
\leq |sn^{s-1}|$ so that $((n+1)^s - n^s)/((n+1)^s n^s) \leq
\abs{s} n^{-1} (n+1)^{-s} < \abs{s} n^{-1-s}$.  The function
$\zeta_{g,\alpha}(s)$ is analytic in ${\rm Re}(s)>0$ as the limit of a
sequence of analytic functions which converge uniformly on a compact
subset of the right half plane. The uniform convergence follows from
Bohr's theorem (see \cite{HardyRiesz}, Theorem 52).
\end{proof}

We do not know what happens for irrational $\alpha$ which are not Diophantine
except if $g$ is a trigonometric polynomial:

\begin{lemma}
\label{dirichlettrig}
If $g$ is a trigonometric polynomial of period 1 with $\int_0^1 g(x) \; dx=0$ 
and $\alpha$ is an arbitrary irrational number, then the abscissa of convergence 
of $\zeta_{g,\alpha}$ is $0$. 
\end{lemma}

\begin{proof}
  If $g$ is a trigonometric polynomial, then $g$ is a coboundary for
  {\bf every} irrational $\alpha$ because $e^{2 \pi i n x} =
  h(x+\alpha)-h(x)$ for $h(x) = e^{2 \pi i n x}/(e^{2 \pi i n x}-1)$.
  In the case of the Clausen function for example and $s=0$, we have
  $\sum_{k=0}^{n-1} \sin(2 \pi k \alpha) = {\rm im}(e^{2 \pi i n
    \alpha}-1)/(e^{2 \pi i \alpha}-1)$.  It follows that the series
  $\zeta(s)$ converges for all trigonometric polynomials $g$, for {\bf
    all} irrational $\alpha \neq 0$ and all ${\rm Re}(s) > 0$.  
\end{proof}

\begin{remark}
A special case is the zeta function 
   $$  \sum_{n=1}^{\infty}  e^{2 \pi i n \alpha}/n^s $$ 
     which is can be written as $\psi_s(e^{2 \pi i \alpha})$, 
     where $\psi_s$ is the {\bf polylogarithm}. Integral representations like
$$ \psi_s(z) = \frac{z}{\Gamma(s)} \int_0^1 [\log(1/t)]^{s-1} \frac{dt}{1-zt}  $$
(see i.e. \cite{Lee97}) show the analytic continuation for $z \neq 1$ rsp. $\alpha \neq 0$. 
It follows that the function $\zeta_{g,\alpha}$ has an analytic continuation to the 
entire complex plane for all irrational $\alpha$ if $g$ is a trigonometric
polynomial. We will use polylogarithms later. 
\end{remark}

\section{The bounded variation case}

The result in the last section had been valid if $\alpha$ satisfies
some Diophantine condition and $g$ is real analytic. If the function $g$ is 
only required to be of bounded variation, then the abscissa of convergence 
can be estimated. Lets first recall some definitions: 

A real number $\alpha$ for
which there exist $C>0$ and $r>1$ satisfying
$$ |\alpha-\frac{p}{q}| \geq \frac{C}{q^{1+r}} $$
for all rational $p/q$ is called {\bf Diophantine of type $r$}. The set of 
real numbers of type $r$ have full Lebesgue measure for all $r>1$ so that 
also the intersection
of all these types have. The {\bf variation} of a function $f$
is $\sup_P  \sum_{i} |f(x_{i+1}) - f(x_i)|$, where the supremum is taken over all
partitions $P=\{x_1, \dots, x_n \}$ of $[0,1]$.

\begin{propo}
  \label{propo4}
  If $\alpha$ is Diophantine of type $r>1$ and if $g$ is of bounded
  variation with $\int_0^1 g(x) \; dx=0$, 
  then the series for $\zeta_{g,\alpha}(s)$ has an abscissa
  of convergence $\sigma_0 \leq (1-1/r)$.
\end{propo}
\begin{proof}
The Denjoy-Koksma inequality 
(see Lemma~\ref{jitomirskaja}) implies that
for any $m,n$, the sum $S_{n,m}=\sum_{k=n}^m g(k\alpha)$ satisfies
the estimate $S_{n,m} \leq C \log(m-n) \; (m-n)^{1-1/r}$, with $C=\Var(g)$. 
Cahen's formula for the abscisse of convergence  gives
$$ \limsup_{n \to \infty} \frac{\log S_{1,n}}{\log(n)} \leq 1-\frac{1}{r} \; . $$
\end{proof} 

\begin{remark}
A weaker result could be obtained directly, without Cahen's formula. 
Choose $\ell_k \to \infty$ such that
$$ U_k = \sum_{j=\ell_k}^{l_{k+1}} [g(j \alpha)/\ell_k^s] \leq C \log(\ell_{k+1}-\ell^k) (\ell_{k+1}-\ell^k)^{1-1/r}/\ell_k^s $$
$$ V_k = \sum_{j=\ell_k}^{l_{k+1}} [ \frac{g(j)}{\ell_k^s} -  \frac{g(j)}{j^s} ] 
\leq \frac{(\ell_{k+1}-\ell_k)^2}{l_{k+1}^s \ell_k^s} ||g|| $$
are summable. Then 
$$  S_k = \sum_{j=\ell_k}^{l_{k+1}} g(j \alpha)/j^s  \leq U_k + V_k $$
is summable. 
\end{remark}

To compare: if $g(T^nx)$ are independent, identically distributed 
random variables with mean $0$ and finite variance $a$, then 
$\sum_{k=1}^n g(T^kx)$ grows by the law of iterated logarithm like $a \sqrt{n} \log \log(n)$ and
the zeta function converges with probability $1$ for ${\rm Re}(s)>1/2$. The following
reformulation of the law of iterated logarithm follows directly from Cahen's formula:  \\

{\it (Law of iterated logarithm)
If $T$ is a Bernoulli shift such that $g(x) = x_0$ produces
independent identically distributed random variables $g(T^nx)$ with finite nonzero variance,
then the Dirichlet series $\zeta_{g,T}$ has the abscissa of convergence $1/2$. 
}

We do not have examples in the almost periodic case yet, where the abscissa of convergence is
strictly between $0$ and $1$ like $\sigma_0=1/2$. 

\section{Analytic continuation}

The {\bf Lerch transcendent} is defined as
$$ L(z,s) = \sum_{n=0}^{\infty} z^n/(n+a)^s  \; . $$
For $z=e^{2\pi i \alpha}$, we get the {\bf Lerch zeta function}
$$ L(\alpha,s) = \sum_{n=0}^{\infty} e^{2\pi i n \alpha}/(n+a)^s  \; . $$
In the  special case $a=1$ we have the {\bf polylogarithm} $L(z,s) = \sum_{n=1}^{\infty} z^n/n^s$.
For $z=1$ and general $a$, we have the {\bf Hurwitz zeta function}, which becomes for
$a=1,z=1$ the {\bf Riemann zeta function}. The following two Lemmas are standard
(see \cite{Ivic}).

\begin{lemma}
For and $|z|=1$, there is an integral representation
$$ L(z,s) = \frac{1}{\Gamma(s)} \int_0^{\infty} \frac{t^{s-1} e^{-at}}{1-z e^{-t}} \; dt  \; . $$
For fixed $|z|=1, z \neq 1$, this is analytic in $s$ for ${\rm Re}(s)>1$. 
For fixed ${\rm Re}(s)>1$, it is analytic in $z$ for $z \neq 1$.
\end{lemma}

\begin{proof}
By expanding $1/(1-z e^{-t}) = \sum_{n=0}^{\infty} z^n/e^{n t}$, the claim is  equivalent to
$$   \frac{1}{\Gamma(s)} \int_0^{\infty} t^{s-1} e^{-at-nt} \; dt = \frac{1}{(n+a)^s} \; . $$
A substitution $u=(a+n) t, du=(a+n) dt$ changes this to 
$$ \frac{1}{\Gamma(s)} \int_0^{\infty} u^{s-1} e^{-u}\frac{1}{(a+n)^s} \, du  = \frac{1}{(n+a)^s}\; . $$
Now use $\Gamma(s) = \int_0^{\infty} u^{s-1} e^{-u} \, du$.

The improper integral is
analytic in $s$ because $|1-z e^{-t}|$ is $\geq \sin({\rm arg}(z))$ for ${\rm
  Re}(z) \geq 0$ and $\geq 1$ 
for ${\rm Re}(z) \leq 0$ and for $\sigma={\rm Re}(s)>1$, we have 
$$ |L(z,s)| 
\leq  \frac{1}{|\Gamma(s)|} \int_0^{\infty} |t^{\sigma-1} e^{-at}| \; dt \frac{1}{|1-z|}  \; . $$

\end{proof}

\begin{lemma}
\label{propo6}
The Lerch transcendent has for fixed $|z|=1, z \neq 1$ and $a>0$ 
an analytic continuation to the entire $s$-plane. 
In every bounded region $G$ in the complex plane, there is a constant $C=C(G,a)$ such that 
$|L(z,s)| \leq C/|z-1|$ and $|(\partial_z)^n L(z,s)| \leq C n!/|z-1|^{n+1}$. 
\end{lemma}

\begin{proof}
For any bounded region $G$ in the complex plane we can find a constant $C$ 
such that
\begin{eqnarray*}
  |\int_0^{\infty} \frac{t^{s-1} e^{-at}}{1-z e^{-t}} \; dt| &\leq&
  (\int_0^{\infty} |t^{s-1} e^{-at}| \; dt) \max_t \frac{1}{|1-z e^{-t}|}  \leq \frac{C}{|1-z|} \; . 
\end{eqnarray*}
Similarly, we can estimate $|\partial_z^n L(z,s)| \leq C n!/|z-1|^n$ for any integer $n>0$. 
The identity 
\begin{equation}
\label{identity}
L(z,s-1) = (a+z \partial_z) L(z,s) 
\end{equation}
allows us to define $L$ for ${\rm Re}(s)<1$: first define $L$ in $0<{\rm Re}(s)<1$ 
by the recursion~(\ref{identity}). Then use the identity~(\ref{identity})
again to define it in the strip $-1<{\rm Re}(s)<0$, then in the strip $-2<{\rm Re}(s)<1$, etc. 
\end{proof}

\begin{remark}
The Lerch transcendent is often written as a function of three variables:
$$  \phi(x,a,s) = \sum_{n=0}^{\infty} e^{2\pi i x}/(n+a)^s \;  . $$
it satisfies the functional equation
\begin{eqnarray*}
 (2\pi)^s \phi(x,a,1-s) &=& \Gamma(s) \exp(2\pi i (s/4-ax)) \phi(x,-a,s) \\
                        &+&  \Gamma(s) \exp(2\pi i(-s/4+a(1-x))) \phi(1-x,a,s) )\; . 
\end{eqnarray*}
See \cite{lerch}. Using the functional equation 
to do the analytic continuation is less obvious. 
\end{remark}

One of the main results in this paper is the following theorem:

\begin{thm}
  \label{analyticcontinuation}
  For all Diophantine $\alpha$ and every real analytic periodic function
  $g$ satisfying $\int g(x) \; dx=0$, the series 
  $\zeta_{g,\alpha}(s) = \sum_{n=1}^{\infty} g(n \alpha)/n^s$ 
  has an analytic continuation to the entire complex plane.
\end{thm}

\begin{proof} 
  The Fourier expansion of $g$ evaluated at $x = n \alpha$ gives 
  $$ g(n \alpha) = \sum_{k=-\infty}^{\infty} c_k e^{2 \pi i n k \alpha} \; . $$ 
  It produces a polylog expansion of $f = \zeta_{g,\alpha}$
  $$ f(s) = \sum_{k=-\infty}^{\infty} c_k L(e^{2\pi i k \alpha},s) $$ 
  with $a=1$.  Because
  $g$ is real analytic, there exists $\delta>0$ such that $|c_k| \leq
  e^{-|k| \delta}$. Since $k \neq 0$ and $\alpha$ is Diophantine, $|e^{2\pi i k
    \alpha}-1| \geq c/|k|^r$ so that $L(e^{2\pi i k \alpha},s) \leq
  |k|^r/c$ and
$$ |f(s)| = \sum_{k=-\infty}^{\infty} |c_k| |L(e^{2\pi i k \alpha},s)| 
       \leq \sum_{k=-\infty}^{\infty} e^{-|k| \delta} k^r/c <\infty\; . $$
\end{proof}

\section{A commutation formula}

For any periodic function $g$, the series 
$$  T(g) = \sum_{n=1}^{\infty} \frac{g(n \alpha)}{n^s} $$
produces for fixed $s$ in the region of convergence a new periodic
function in $\alpha$. For fixed $\alpha$ it is a 
Dirichlet series in $s$. The Clausen function is 
$T(\sin 2 \pi x)$ and the polylogarithm is $T(\exp(2 \pi i x))$.  We may then
consider a new almost periodic Dirichlet series generated by this
function, defined by $T(T(g)(s))(t)) = T(T(g))(s,t)$. \\

The following commutation formula can be useful to extend the 
domain, where Dirichlet series are defined:

\begin{lemma}[Commutation formula]
\label{commutationformula}
For ${\rm Re}(s)>1$ and ${\rm Re}(t)>1$, we have
$$   T(T(g))(s,t) = T(T(g))(t,s)) \; . $$
\end{lemma}

\begin{proof}
We have 
$$ T(g)(s) = \sum_{n=1}^\infty \frac{g(n\alpha)}{n^{s}} \; ,$$
which we regard as a periodic function in $\alpha$. It is continuous in
$\alpha$ if evaluated for fixed ${\rm Re}(s)>1$. 
Then where our sum converges absolutely (which holds at least for $s,t>1$),
\begin{align*}
T(T(g))(s,t) &= \sum_{m=1}^\infty \frac{T(g)(s)(m \alpha)}{n^t} \\
&= \sum_{m=1}^\infty \sum_{n=1}^\infty \frac{1}{n^s m^t}
g(mn\alpha) \\
&= \sum_{n=1}^\infty \sum_{m=1}^\infty \frac{1}{n^s m^t}
g(mn\alpha) = T(T(g))(t,s).
\end{align*}
\end{proof}

In fact, the double sum can be expressed as a single sum using the
{\bf divisor sum function} $\sigma_t(k) = \sum_{m | k} m^t$:
\begin{align}
T(T(g))(s,t) &=  \sum_{m=1}^\infty \sum_{n=1}^\infty \frac{1}{n^{s} m^t}
g(mn\alpha) = \sum_{k=1}^\infty \left( \left( \sum_{m \mid k}
    \frac{m^{t-s}}{k^t} \right) g(k\alpha) \right) \notag\\
&= \sum_{k=1}^\infty \frac{\sigma_{t-s}(k)}{k^t} g(k\alpha) \;. \label{multfourier}
\end{align}

We can formulate this as follows: for $s,t>1$, we have
$$  \zeta_{g_s,\alpha}(t) = \zeta_{g_t,\alpha}(s) \; . $$

For example, evaluating the almost periodic Dirichlet series for the
periodic function $g_3(x) = \sum_{k=1}^{\infty} (1/k^3) \sin(2\pi k x)$ at $s=5$ is the
same as evaluating the almost periodic Dirichlet series of the periodic
function $g_5(x) = \sum_{k=1}^{\infty} (1/k^5) \sin(k x)$ and evaluating it
at $s=3$. But since $g_3(x)$ is of bounded variation, the Dirichlet 
series $T(g_3)$ has
an analytic continuation to all ${\rm Re}(s)>0$ and $T(g_3)(0.5)$ for example is
defined if $\alpha$ is Diophantine of type $1<r<2$. 
The commutation formula
allows us to define $T(g_{0.5})(3) = \zeta_{g_{0.5}}(3)$ as $\zeta_{g_{3}}(0.5)$,
even so $g_{0.5}$ is not in $L^2(T^1)$. 

\begin{remark}
The commutation formula generalizes. 
The irrational rotation $x \mapsto x + \alpha$ on $X=S^1$ can be replaced by a 
general topological dynamical system on $X$.
\end{remark}

In the particular case that $f$ is the Clausen function $g=T(\sin)$,
the expression~(\ref{multfourier}) is itself a Fourier series, whose
coefficients are the divisor function $\sigma$.
We have $g(x) = \sum_{n=1}^{\infty} \frac{1}{n^t} \sin(2 \pi n x)$ and $t>s$:
the value of $\zeta(s)$ as a function of $\alpha$ is the Fourier transform on $l^2(\Z)$ 
of the multiplicative arithmetic function $h(k) = \sigma_{t-s}(k)/k^s$. 
For $t=2,s=1$ for example, we get 
$$  \sum_{n} \frac{g(n \alpha)}{n}  = \sum_{k=1}^{\infty} \frac{\sigma(k)}{k} \sin(2\pi k \alpha) $$
if $g$ is the function with Fourier coefficients $1/n^2$. In that
case, the Fourier coefficients of
the function has the multiplicative function $\sigma(n)/n$ (called the {\bf index} of $n$) as coefficients. 
For $t=1$, we have
$\sum_{n=1}^{\infty} \frac{1}{n} \sin(2\pi n x) = 1/2-(x \; {\rm mod}  \; 1)$ and
for odd integer $t>1$, the function 
$$  g_k(x) = \sum_{k=1}^{\infty} \sin(2\pi k x)/k^t $$ 
is a {\bf Bernoulli polynomial}. \\ 

For $t=s$, we get 
$$ \sum_{n} \frac{g(n \alpha)}{n^s}  = \sum_{k=1}^{\infty} \frac{d(k)}{k^s} \sin(2\pi k \alpha) \; , $$
where $d(k)$ is the number of divisors of $k$. These sums converge
absolutely for ${\rm Re}(s)>1$.  If $t$ is a positive odd integer,
$g$ is a Bernoulli polynomial (e.g.\ for $t=3$, we have 
$g(x) = (\pi^3/3) (x-3 x^2+ 2 x^3)$). For $s=2$ (and still $t=3$), we have
$$  \sum_{n} \frac{g(n \alpha)}{n^2} = \sum_{k=1}^{\infty} \frac{\sigma(k)}{k^2} \sin(2\pi k \alpha) \; . $$

The functions $f_{s,t}(\alpha) = T(T(\sin))(s,t)$, regarded as periodic functions of
$\alpha$, may be related by an identity of Ramanujan~\cite{wilson}.
Applying Parseval's theorem to these Fourier series, one can deduce
\begin{align*}
  \int_0^1 f_{s,t}(\alpha) f_{u,v}(\alpha) \, d\alpha 
     &= \sum_{n=1}^\infty \frac{\sigma_{t-s}(n)}{n^t} \frac{\sigma_{v-u}(n)}{n^v} \\
     &= \frac{\zeta(s+u) \zeta(s+v) \zeta(t+u) \zeta(t+v)}{\zeta(s+t+u+v)} \; .
\end{align*}

\section{Unbounded variation}

If $g$ fails to be of bounded variation, the previous results do not apply. Still, 
there can be boundedness for the Dirichlet series if ${\rm Re}(s)>0$. 
The example $g(x) = \log|2-2 \cos(2 \pi x)|=2 \log|e^{2 \pi i x}-1|$ 
appears in the context of KAM theory and was the starting point of our investigations. 
The product $\prod_{k=1}^n | 2 \cos(2\pi k \alpha) -2|$ is the determinant of a
truncated diagonal matrix representing the Fourier transform of the Laplacian 
$L(f) = f(x+\alpha) - 2 f (x) + f(x-\alpha)$ on $L^2(\T)$. \\

The function $g(x)$ has mean $0$ but it is not bounded and 
therefore has unbounded variation. Numerical experiments indicate
however that at least for many $\alpha$ of constant type,
$|\sum_{k=0}^{n-1} g(k \alpha)| \leq C \log(n)$. We can only show:

\begin{propo} 
\label{propo7}
If $\alpha$ is Diophantine of type $r>1$, then for $g(x) = \log|2-2 \cos(2 \pi x)|$,
$$  \sum_{n=0}^{k-1} g(n \alpha) \leq C k^{1-1/r} (\log(k))^2 $$
and $\zeta(s)$ converges for ${\rm Re}(s)>0$. 
\end{propo}
\begin{proof}
$\prod_{j=1}^{q-1} |e^{2\pi i j/q}-1| = q$ because 
$$ \prod_{j=1}^{q-1} |e^{2\pi i j/q}  - z| = \frac{z^q-1}{z-1} = \sum_{j=0}^{q-1} z^j $$
which gives $\prod_{j=1}^{q-1} |e^{2\pi i  j/q}-1| = q$. 
Define $\log^M(x) = \max( -M,\log(x))$. Now
$$ \sum_{j=1}^{q-1} \log^M(|e^{ 2\pi i j \alpha}-1|)  \leq M \log(q)  $$
for all $q$ and also for general $k$ by the classical Denjoy-Koksma inequality. Choose $M = 3 \log(q)$. Then 
the set $Y_M = \{ \log(|\sin(2 \pi x)|) < -M \; \} = 
               \{ |\sin(2 \pi x)| < e^{-M} \; \}  \subset \{ |x|<2 e^{-M} = 2 q^{-3} \}$. 
The finite orbit $\{ k \alpha \}_{k=1}^{q-1}$ never hits that set and the sum is the same when 
replacing $\log^M$ with the untruncated $\log$. We get therefore 
$$ \sum_{j=1}^{q-1} \log(|e^{ 2\pi i j \alpha}-1|) = \sum_{j=1}^{q-1} \log^M(|e^{ 2\pi i j \alpha}-1|)  \leq 3 \log(q)^2  \; . $$
The rest of the proof is the same as for the classical Denjoy-Koksma
inequality.
\end{proof}

The Denjoy-Koksma inequality is treated in \cite{Her79,CFS}.
Here is the exposition as found in \cite{Jit99}. 

\begin{lemma}[Jitomirskaja's formulation of Denjoy-Koksma]
\label{jitomirskaja}
Assume $\alpha$ is Diophantine of type $r>1$ and $g$ is of bounded 
variation and $\int_0^1 g(x) \; dx=0$. Then 
$S_k = \sum_{n=1}^k g(n\alpha)$ satisfies
$$ |S_k| \leq  C k^{1-1/r} \log(k) {\rm Var}(g)  \; . $$
If $\alpha$ is of constant type and $g$ is of bounded variation and
$\int_0^1 g(x) \; dx=0$, then $S_k \leq C \log(k)$.
\end{lemma}
\begin{proof}
(See  \cite{Jit99}, Lemma 12). 
If $p/q$ is a periodic approximation of $\alpha$, then
$$ |S_q| \leq {\rm Var}(f)   \; . $$
To see this, divide the circle into $q$ intervals centered at the points
$y_m=m p/q$. These intervals have length $1/q \pm O(1/q^2)$ and each interval
contains exactly one point of the finite orbit $\{ y_k = k \alpha \}_{k=1}^q$. 
Renumber the points so that $y_m$ is in $I_m$.
By the intermediate value theorem, there exists a Riemann sum
$\frac{1}{q} \sum_{i=0}^{q-1} f(x_i) = \int f(x) \; dx =0$ for which every
$x_i$ is in an interval $I_i$ (choosing the point $x_i=\min_{x \in I_i} f(x)$ gives
an lower and $x_m = \max_{x \in I_m} f(x)$ gives an upper bound).
If $\sum_{j=0}^{q-1} f(y_j) - f(x_j)  
  \leq \sum_{j=0}^{q-1} |f(y_j) - f(x_j)| + |f(x_j) - f(y_{j+1})| \leq {\rm Var}(f)$. \\
Now, if $q_m \leq k \leq q_{m+1}$ and
$k = b_m q_m + b_{m-1} q_{m-1} + \dots + b_1 q_1 + b_0$, then
$$ S_k \leq  (b_0 + \dots + b_m) {\rm Var}(f) \leq \sum_{i=0}^{m} \frac{q_{i+1}}{q_i} {\rm Var}(f) $$
because $b_j \leq q_{j+1}/q_j$. \\
If $\alpha$ is of constant type then $\frac{q_{i+1}}{q_i}$ is bounded and
$m<2 \log(k)/\log(2)$ implies $S_k \leq  (2 \log(k)/\log(2)) {\rm Var}(f)$.  \\
If $\alpha$ is Diophantine of type $r>1$, then $||q \alpha|| \leq c/q^r$
and $q_{i+1} \leq q_i^r/c$ which implies $q_{i+1}/q_i < q_{i+1}^{1-1/r}/c^{1/r}$ and so
$$ |S_k|  \leq (c^{-1/r} \sum_{i=1}^m q_i^{1-1/r} + \frac{k}{q_m}) {\rm Var}(f)
          \leq (c^{-1/r} m q_m^{1-1/r} + \frac{k}{q_m}) {\rm Var}(f) \; . $$
The general fact $m \leq 2 \log(q_m)/\log(2) \leq 2 \log(k)/\log(2)$ deals with the first term.
The second term is estimated as follows: from $k \leq q_{m+1} \leq q_m^r/c$, we have
$q_m \geq (c k)^{1/r}$ and $k/q_m \leq c^{-1/r} k^{1-1/r}$.
\end{proof}

\remark{
One knows also $|S_n| \leq C \log(n)^{2+\epsilon}$ if the continued
fraction expansion $[a_0,a_1,...]$ of $\alpha$ satisfies $a_m< m^{1+\epsilon}$
eventually. See \cite{Guillotin}.
}
\section{Questions}
We were able to get entire functions $\zeta(s) = \sum_n g(n
\alpha)/n^s$ for rational $\alpha$ and Diophantine $\alpha$. What
happens for Liouville $\alpha$ if $g$ is not a trigonometric polynomial?
What happens for more general $g$? 

For every $g$ and $s$ we get a function $\alpha \to \zeta_{g,\alpha}(s)$. For
$$ g(x) = \sum_{n=1}^{\infty} \frac{1}{n} \sin(2 \pi n x) $$ 
and Diophantine $\alpha$, where
$$ h(\alpha) = \zeta_{g,\alpha}(1) = \sum_{n=1}^{\infty} \frac{d(n)}{n} \sin(2 \pi n \alpha) \; , $$ 
we observe a self-similar nature of the graph.
Is the Hausdorff dimension of the graph of $h$ not an integer?  

One can look at the problem for more general dynamical systems. Here is an example:
for periodic Dirichlet series generated by an ergodic translation on a two-dimensional 
torus with a vector $(\alpha,\beta)$, where $\alpha,\alpha/\beta$ are irrational, the series is
$$  \sum_{n=1}^{\infty}  \frac{g(n \alpha,n \beta)}{n^s}  \; . $$
In the case $s=0$, this leads to the Denjoy-Koksma type problem to estimate the growth
rate of the random walk
$$ \sum_{n=1}^{\infty} g(n \alpha,n \beta) \;  $$
which is more difficult due to the lack of a natural continued
fraction expansion in two dimensions.  In a concrete example like
$g(x,y) = \sin(2 \pi x y)$, the question is, how fast the sum
$$  S_k = \sum_{n=1}^k \sin(2\pi n^2 \gamma)  $$ 
grows with irrational $\gamma=\alpha \beta$. Numerical experiments
indicate subpolynomial growth that $S_k = O(\log(k)^2)$ would hold for 
Diophantine $\gamma$ and suggest the abscissa of convergence of the 
Dirichlet series
$$  \zeta(s) = \sum_{n=1}^{\infty} \frac{\sin(2 \pi n^2 \gamma) }{n^s} \;  $$
is $\sigma_0=0$. This series is of some historical interest since Riemann knew 
in 1861 (at least according to Weierstrass \cite{kahane64}) that 
$h(\gamma) = \sum_{n=1}^{\infty} \frac{\sin(2 \pi n^2 \gamma) }{n^2}$ is
nowhere differentiable. 

\bibliographystyle{plain}

\end{document}